\documentclass{amsart}

\usepackage{color,multicol}
\usepackage{amssymb}
\theoremstyle{plain}
\newtheorem{theorem}{Theorem}[section]
\newtheorem{lemma}[theorem]{Lemma}

\newtheorem{prop}[theorem]{Proposition}

\theoremstyle{definition}

\theoremstyle{remark}

\begin{document}

\title[Christoffel functions on the Unit Ball with
a Mass on the Sphere]{Asymptotic behaviour of the Christoffel functions on the Unit Ball in the presence
of a Mass on the Sphere}

\author[C. Mart\'{\i}nez]{Clotilde Mart\'{\i}nez}
\address[C. Mart\'{\i}nez]{{Departamento de Matem\'atica Aplicada\\
Universidad de Granada\\
18071 Granada, Spain}}\email{clotilde@ugr.es}
\author[M. A. Pi\~{n}ar]{Miguel A. Pi\~{n}ar}
\address[M. A. Pi\~{n}ar]{{Departamento de Matem\'atica Aplicada\\
Universidad de Granada\\
18071 Granada, Spain}}\email{mpinar@ugr.es}

\begin{abstract}
{ We present}   a family of mutually orthogonal polynomials on the unit ball with respect to an inner
product which includes a mass uniformly distributed on the sphere. First, connection formulas relating
these multivariate orthogonal polynomials and the classical ball polynomials are obtained. Then,
using the representation formula for these polynomials in terms of spherical harmonics
analytic properties will be deduced. {Finally, we analyze the asymptotic behaviour of the Christoffel functions.}
\end{abstract}



\maketitle

\section{Introduction}
\setcounter{equation}{0}

Classical orthogonal polynomials on the unit ball $\mathbb{B}^d$ of $\mathbb{R}^d$
correspond to the classical inner product
$$
  \langle f, g\rangle_\mu = \frac{1}{\omega_\mu}\int_{\mathbb{B}^d} f(x) g(x) W_\mu(x) dx,
$$
where $W_\mu(x) = (1-\|x\|^2)^\mu$ on $\mathbb{B}^d$, $\mu > -1$, and
$\omega_\mu$ is a normalizing constant such that $\langle 1 ,1 \rangle_\mu = 1$.

\bigskip

In the present paper, we consider orthogonal polynomials with respect to the inner product
$$
 \langle f ,g \rangle_{\mu}^{\lambda}
     =  \frac{1}{\omega_\mu} \int_{\mathbb{B}^d} f(x) g(x)  W_\mu(x)  dx +
    \frac{\lambda}{\sigma_{d-1}}  \int_{\mathbb{S}^{d-1}} f(\xi) g(\xi)d\sigma(\xi),
$$
where $\lambda>0$, $d \sigma$ is the surface measure on $\mathbb{S}^{d-1}$ and $\sigma_{d-1}$ denotes the sphere area.

\bigskip

Using spherical polar coordinates, we shall construct a family of mutually orthogonal polynomials
with respect to $ \langle \cdot ,\cdot \rangle_{\mu}^{\lambda}$, which depends
on a sequence of orthogonal polynomials of one variable, namely the Krall polynomials \cite{Krall1940}. 
This sequence of orthogonal polynomials can be expressed
in terms of Jacobi polynomials. In a previous paper (see \cite{MartinezPinar2016})
we have shown that the multivariate polynomials orthogonal with respect to 
the inner product $ \langle \cdot ,\cdot \rangle_{\mu}^{\lambda}$ satisfy a fourth order partial 
differential equation.

Standard techniques provide us explicit connection formulas relating classical multivariate
ball polynomials and our family of orthogonal polynomials. The explicit representations
for the norms and the kernels will be obtained.

\bigskip

A very interesting open problem in the theory of multivariate orthogonal polynomials is that of
finding asymptotic estimates for the Christoffel functions, because these estimates are related
to the study of the convergence of the Fourier series. Asymptotics for Christoffel functions associated to the
classical orthogonal polynomials on the ball were obtained by Y. Xu in 1996 (see \cite{Xu1996}).
Recently, more general results on the asymptotic behaviour of the Christoffel functions were
established by A. Kro\'o and D. Lubinsky \cite{KrooLubinsky2013A,KrooLubinsky2013B} 
in the context of universality. Those results include estimates in a  quite general case where 
the orthogonality measure satisfies some regularity conditions. In fact, they provide the 
ratio asymptotic for the Christoffel functions corresponding to two regular measures 
supported on the same compact set $D \in \mathbb{R}^d$, in particular, the ratio of the
Christoffel functions converges uniformly on any compact subset of the interior of $D$.

Since our orthogonal polynomials does not fit into the above mentioned situation, the asymptotic of the
Christoffel functions deserves special attention. Not surprisingly, our results show that
in any compact subset of the interior of the unit ball Christoffel functions 
behave exactly as in the classical case, see Theorem~\ref{th2}. On the sphere the situation
is quite different and we can perceive the influence of the mass $\lambda$, see Theorem~\ref{th1}.

A similar problem on a Sobelev context where the mass on the sphere was replaced by the normal derivatives has
been recently considered in \cite{DelgadoFLPP2016}.

\bigskip

The paper is organized as follows. In the next section, we state the background materials on orthogonal
polynomials on the unit ball and spherical harmonics that we will need later.
In Section 3, using spherical polar coordinates we construct explicitly a family
of mutually orthogonal polynomials with respect to $ \langle \cdot ,\cdot \rangle_{\mu}^{\lambda}$.
Those polynomials are given in terms of spherical harmonics and a family of univariate
orthogonal polynomials in the radial part, their properties are studied in Section 4.
In Section 5, we deduce explicit connection formulas relating classical multivariate
ball polynomials and our family of orthogonal polynomials.
Moreover, an explicit representation for the kernels is obtained. The asymptotic
behaviour of the corresponding Christoffel functions is studied in Section 6.

\section{Classical orthogonal polynomials on the ball}
\setcounter{equation}{0}

In this section we describe background materials on orthogonal polynomials and spherical harmonics. 
The first subsection collects some properties on the Jacobi polynomials. 
The second subsection recalls the basic results on spherical harmonics
and classical orthogonal polynomials on the unit ball.

\bigskip

\subsection{Classical Jacobi polynomials}

For $\alpha, \beta > -1$, Jacobi polynomials $P_n^{(\alpha,\beta)}(t)$ \cite{Szego1975} are orthogonal 
with respect to the Jacobi inner product
$$
\left(f,g\right)_{\alpha,\beta}=\int_{-1}^1f(t)\, g(t)\, (1-t)^\alpha(1+t)^{\beta} dt.
$$
and satisfy
\begin{equation} \label{jac-norm}
P_n^{(\alpha,\beta)}(1) = \binom{n+\alpha}{n} = \frac{(\alpha+1)_n}{n!}.
\end{equation}
The squares of the $L^2$ norms are given by
\begin{equation}\label{normJ}
h_n^{(\alpha,\beta)} :=  \left(P_{n}^{(\alpha, \beta)},P_{n}^{(\alpha, \beta)}\right)_{\alpha,\beta}
= \frac{2^{\alpha+\beta+1}\,\Gamma(n+\alpha+1)\,\Gamma(n+\beta+1)}{(2n+\alpha+\beta+1)\,n!\,
\Gamma(n+\alpha+\beta+1)}.
\end{equation}
Furthermore, to Jacobi polynomials we will use the corresponding kernel polynomials
defined as
\begin{equation*}\label{KernelJac}
K_n^{(\alpha,\beta)}(t,u) = \sum_{k=0}^n \, \frac{P_k^{(\alpha,\beta)}(t)\,
P_k^{(\alpha,\beta)}(u)}{h_k^{(\alpha,\beta)}},
\end{equation*}
which are symmetric functions. 
It is well known (see \cite[p. 71]{Szego1975}) that
\begin{align}
K_{n}^{(\alpha,\beta)}(t,1) &= \frac{2^{-\alpha-\beta-1}}{\Gamma(\alpha+1)}\frac{\Gamma(n+\alpha+\beta+2)}{\Gamma(n+\beta+1)}\,
P_{n}^{(\alpha+1,\beta)}(t), \label{K(t,1)}\\
K_{n}^{(\alpha,\beta)}(1,1) &= \frac{2^{-\alpha-\beta-1}}{\Gamma(\alpha+1)}
\frac{\Gamma(n+\alpha+\beta+2)}{\Gamma(n+\beta+1)}\frac{\Gamma(n+\alpha+2)}{\Gamma(n+1)\,
\Gamma(\alpha+2)}.  \label{K(1,1)}
\end{align}

\bigskip

\subsection{Orthogonal polynomials on the unit ball and spherical harmonics}

Let $\Pi^d$ be the space of polynomials in $d$ real variables. For a given
non negative integer $n$,
let $\Pi_n^d$ denote the linear space of polynomials in several variables of (total) degree
at most $n$ and let $\mathcal{P}_n^d$ be the space of homogeneous polynomials of degree
$n$. 

The unit ball and the unit sphere in $\mathbb{R}^d$ are
denoted, respectively, by
$$
\mathbb{B}^d =\{x\in \mathbb{R}^d: \|x\| \leqslant 1\} \qquad \textrm{and} \qquad 
\mathbb{S}^{d-1}=\{\xi\in \mathbb{R}^d: \|\xi\| = 1\}.
$$
where $\|x\|$ denotes as usual the Euclidean norm of $x$.

For $\mu \in \mathbb{R}$, the weight function $W_\mu(x) = (1-\|x\|^2)^{\mu}$ is integrable on the unit ball if $\mu > -1$. 
Consider the inner product
$$
   \langle f,g \rangle_\mu = \frac{1}{\omega_\mu} \int_{{\mathbb{B}^d}}\, f(x)\, g(x) \, W_\mu(x) \, dx,
$$
where  ${\omega_\mu}$ is the normalization constant of $W_\mu$ given by

\begin{equation*}\label{omegamu}
\omega_\mu := \int_{{\mathbb{B}^d}}\, W_\mu(x) \, dx = \frac{\pi^{d/2}\Gamma(\mu+1)}{\Gamma(\mu + 1 + d/2)}.
\end{equation*}
 

A polynomial $P \in \Pi_n^d$ is called \textbf{orthogonal} with respect to $W_\mu$ on the ball if
$\langle P, Q\rangle_\mu =0$ for all $Q \in \Pi_{n-1}^d$, that is, if it is orthogonal to all polynomials
of lower {degree}. Let $\mathcal{V} _n^d(W_\mu)$ denote the space of orthogonal polynomials of total
degree $n$ with respect to $W_\mu$. It is well known that
\[
\dim \Pi_n^d = \binom{n+d}{d} \quad \textrm{and} \quad  
\dim \mathcal{V} _n^d(W_\mu) = \binom{n+d-1}{n}.
\]

For $n\ge 0$, let $\{P^n_{\nu}(x) : |\nu|=n\}$ denote a basis of $\mathcal{V}_n^d(W_\mu)$.
Notice that every element of $\mathcal{V} _n^d(W_\mu)$ is orthogonal to polynomials of lower degree. If the
elements of the basis are also orthogonal to each other, that is,
$\langle P_\nu^n, P_\eta^n \rangle_\mu=0$ whenever $\nu \ne \eta$,
we call the basis \textbf{mutually orthogonal}. If, in addition,
$\langle P_\nu^n, P_\nu^n \rangle_\mu =1$, we call the basis \textbf{orthonormal}.

\bigskip

Harmonic polynomials of degree $n$ in $d$--variables are polynomials in $\mathcal{P} _n^d$ satisfying
the Laplace equation $\Delta Y = 0$, where
$$
\Delta = \frac{\partial^2}{\partial x_1^2} + \ldots + \frac{\partial^2}{\partial x_d^2}
$$
is the usual Laplace operator.

Let $\mathcal{H}_n^d$ be the space of harmonic polynomials
of degree $n$. It is well known that
$$
         a_n^d: = \dim \mathcal{H}_n^d = \binom{n+d-1}{d-1} - \binom{n+d-3}{d-1}.
$$
Spherical harmonics are the restriction of harmonic polynomials to the unit sphere. If $Y \in
\mathcal{H}_n^d$, {then} in spherical--polar {coordinates}
$x = r \xi$, $ r \geqslant 0$ and $\xi \in \mathbb{S}^{d-1}$, we get 
\[
Y(x) = r^n Y(\xi),
\] 
so that $Y$ is uniquely determined by its restriction to the sphere. We shall also use 
$\mathcal{H}_n^d$ to denote the space of spherical harmonics of degree $n$.

\bigskip

If $d \sigma$ denotes the surface measure then the surface area is given by 
\begin{equation*} \label{sigma_d}
  \sigma_{d-1} := \int_{\mathbb{S}^{d-1}} d\sigma = \frac{2\, \pi^{d/2}}{\Gamma(d/2)}.
\end{equation*}
Spherical harmonics of different degrees are orthogonal with respect to the inner product
$$
   \langle f, g \rangle_{\mathbb{S}^{d-1}} = \frac{1}{\sigma_{d-1}} \int_{\mathbb{S}^{d-1}} f(\xi) g(\xi) d\sigma(\xi). 
$$

\bigskip

Since the weight function $W_\mu$ is rotationally invariant, 
in spherical--polar coordinates $x = r \xi$, $ r \geqslant 0$ and $\xi \in \mathbb{S}^{d-1} $, 
a mutually orthogonal basis of $\mathcal{V}_n^d(W_\mu)$ can be shown in terms of Jacobi 
polynomials and spherical harmonics (see, for instance, \cite{DunklXu2014}).

\begin{prop} \label{baseP}
For $n \in \mathbb{N}_0$ and $0 \le j \le n/2$, let $\{Y_\nu^{n-2j}: 1\le \nu\le a_{n-2j}^d\}$ be
an orthonormal basis for $\mathcal{H}_{n-2j}^d$. Let us denote $\beta_{n-2j} = n-2j + \frac{d-2}{2}$ 
and define 
$$
P_{j,\nu}^{n}(x) := P_{j}^{(\mu,  \beta_{n-2j})}(2\,\|x\|^2 -1)\, Y_\nu^{n-2j}(x). 
$$
Then the set $\{P_{j,\nu}^{n}(x): 1 \le j \le n/2, \,1 \le \nu \le a_{n-2j}^d \}$ is a mutually
orthogonal basis of $\mathcal{V}_n^d(W_\mu)$. More precisely, 
$$
\langle P_{j,\nu}^{n}(x), P_{k,\eta}^{m}(x)\rangle_\mu =  H_{j,\nu}^{n,\mu}  \delta_{n,m}\,\delta_{j,k}\,\delta_{\nu,\eta},
$$
where $ H_{j,\nu}^{n,\mu}$ is given by 
$$ 
 H_{j,\nu}^{n,\mu}: = \frac{(\mu +1)_j (\frac{d}{2})_{n-j} (n-j+\mu+ \frac{d}{2})}
    { j! (\mu+\frac{d}{2}+1)_{n-j} (n+\mu+ \frac{d}{2})} =  \frac{c_\mu^d}{2^{n-2j}} h_j^{(\mu,\beta_{n-2j})},
$$ 
with 
$$
c_\mu^d = \frac{1}{2^{\mu+\frac{d}{2}+1}} \frac{\sigma_{d-1}}{\omega_\mu}.
$$
\end{prop}

\bigskip

\section{An inner product on the unit ball with an extra spherical term} 
\setcounter{equation}{0}

Let us define the inner product
\begin{equation*} \label{main-ip}
 \langle f ,g \rangle_{\mu}^{\lambda}  =  \frac{1}{\omega_\mu} \int_{\mathbb{B}^d} f(x) g(x)  W_\mu(x)  dx +
    \frac{\lambda}{\sigma_{d-1}}  \int_{\mathbb{S}^{d-1}} f(\xi) g(\xi) d\sigma,
\end{equation*}  
where $\lambda>0$.
As a consequence of the central symmetry of the inner product, we can use a procedure
analogous to the construction described in Proposition \ref{baseP} to obtain a basis of $\mathcal{V}_n^d(W_\mu,\lambda)$, the linear 
space of orthogonal polynomials of exact degree $n$ with respect to 
$\langle \cdot, \cdot \rangle_{\mu}^{\lambda}$. This time,
the radial parts constitute a sequence of polynomials in one variable related to Jacobi polynomials.  

\begin{theorem} \label{thm-3.1}
For $n \in \mathbb{N}_0$ and $0 \le j \le n/2$, let $\{Y_\nu^{n-2j}: 1\le \nu\le a_{n-2j}^d\}$ be
an orthonormal basis for $\mathcal{H}_{n-2j}^d$.
Let us denote 
$$
M_{n-2j} = \lambda \frac{2^{n-2j}}{c_{\mu}^d}.
$$
Let $q_{j}^{(\mu, \beta_{n-2j},M_{n-2j})}(t)$ be the $j$-th orthogonal polynomial with respect to
\begin{equation} \label{eq:op1d}
(f,g)_{\mu,\beta_{n-2j}}^{M_{n-2j}} = 
\int_{-1}^1 f(t)g(t)(1-t)^{\mu} (1+t)^{\beta_{n-2j}} dt + M_{n-2j} f(1) g(1), 
\end{equation}
and having the same leading coefficient as the Jacobi polynomial $P_j^{(\mu, \beta_{n-2j})}(t)$.
Then the polynomials
\[
Q_{j,\nu}^{n}(x) = q_{j}^{(\mu, \beta_{n-2j},M_{n-2j})}(2\,\|x\|^2 -1)\, Y_\nu^{n-2j}(x),
\]
with $1 \le j \le n/2, \,1 \le \nu \le a_{n-2j}^d$ constitute a \textbf{mutually
orthogonal basis of $\mathcal{V}_n^d(W_\mu,\lambda)$}. That is,
$$
\langle Q_{j,\nu}^{n}(x), Q_{k,\eta}^{m}(x)\rangle_{\mu}^{\lambda} =  \tilde{H}_{j,\nu}^{n}  \delta_{n,m}\,\delta_{j,k}\,\delta_{\nu,\eta},
$$
where $  \tilde{H}_{j,\nu}^{n}$ is given by 
$$ 
  \tilde{H}_{j,\nu}^{n}: =  \frac{c_\mu^d}{2^{n-2j}}  \tilde{h}_j^{(\mu,\beta_{n-2j},M_{n-2j})},
$$ 
with 
$$
\tilde{h}_j^{(\mu,\beta_{n-2j},M_{n-2j})} = (q_{j}^{(\mu, \beta_{n-2j},M_{n-2j})},
q_{j}^{(\mu, \beta_{n-2j},M_{n-2j})})_{\mu,\beta_{n-2j}}^{M_{n-2j}}.
$$
\end{theorem}

\begin{proof}
The proof of this theorem uses the following well known identity
\begin{equation}\label{changevar}
        \int_{\mathbb{B}^d} f(x) dx = \int_0^1 r^{d-1}\int_{\mathbb{S}^{d-1}}f(r\,\xi)\, d\sigma (\xi)\,dr
\end{equation}
that arises from the spherical--polar coordinates $x=r\,\xi$,$r=\|x\|$, $\xi\in \mathbb{S}^{d-1}$.

In order to check the orthogonality, we need to compute the product
\begin{equation} \label{pr1}
\langle Q_{j,\nu}^n, Q_{k,\eta}^m \rangle_\mu^{\lambda} = \frac{1}{\omega_\mu} \int_{\mathbb{B}^d} Q_{j,\nu}^n(x) Q_{k,\eta}^m(x) W_\mu(x)\, dx + \frac{\lambda}{\sigma_{d-1}} \int_{\mathbb{S}^{d-1}} Q_{j,\nu}^n(\xi) Q_{k,\eta}^m(\xi) \, d\sigma(\xi). 
\end{equation}

Let us start with the computation of the first integral.
\[
I_1=\frac{1}{\omega_\mu} \int_{\mathbb{B}^d} Q_{j,\nu}^n(x) Q_{k,\eta}^m(x) W_\mu(x)\, dx.
\]
To simplify our notations, we will write $q_{j}^{(\mu, \beta_{n-2j},M_{n-2j})} = q_{j}^{({n-2j})}$. Using polar coordinates, relation \eqref{changevar}, and the orthogonality of the
spherical harmonics we obtain
\begin{align*}
I_1 =& \frac{\sigma_{d-1}}{\omega_\mu} \int_0^1 q_{j}^{({n-2j})}(2r^2-1) q_{k}^{({m-2k})}(2r^2-1)(1-r^2)^{\mu}\,  r^{n-2j+m-2k+d-1}  \, dr
\\
& \times \, \delta_{n-2j,m-2k} \delta_{\nu\eta}
\\
=& \frac{\sigma_{d-1}}{\omega_\mu} \int_0^1 q_{j}^{({n-2j})}(2r^2-1) q_{k}^{({m-2k})}(2r^2-1) (1-r^2)^{\mu}\,  r^{2(n-2j)+d-1} \, dr
\\
& \times \, \delta_{n-2j,m-2k} \delta_{\nu\eta}.
\end{align*}
Finally, the change of variables $t=2r^2-1$ moves the integral to the interval $[-1,1]$,
\begin{align}  \label{i1}
I_1 =& \frac{c_\mu^d}{2^{n-2j}}  \int_{-1}^1  q_{j}^{({n-2j})}(t)  q_{k}^{({n-2j})}(t) (1-t)^{\mu} (1+t)^{\beta_{n-2j}}  dt \, \delta_{n-2j,m-2k} \delta_{\nu\eta}. 
\end{align}
For the second integral in \eqref{pr1} we get
\begin{align} \label{i2}
I_2 =& \frac{\lambda}{\sigma_{d-1}}  \int_{\mathbb{S}^{d-1}} Q^n_{j,\nu}(\xi) Q^m_{k,\eta}(\xi) d\sigma(\xi)
\\
=& \frac{\lambda}{\sigma_{d-1}}
q_{j}^{({n-2j})}(1) q_{k}^{({m-2k})}(1)
\, \int_{\mathbb{S}^{d-1}} Y^{n-2j}_{\nu}(\xi) Y^{m-2k}_{\eta}(\xi) d\sigma(\xi) \notag \\ 
=& \lambda q_{j}^{({n-2j})}(1) q_{k}^{({m-2k})}(1)
\, \delta_{n-2j,m-2k} \, \delta_{\nu,\eta}. \notag
\end{align}

To end the proof, we just have to put together \eqref{i1} and \eqref{i2} to get the value of \eqref{pr1}
in terms of the inner product \eqref{eq:op1d} as
\begin{align*}
\langle Q_{j,\nu}^n, Q_{k,\eta}^m \rangle_\mu^\lambda &= \frac{c_\mu^d}{2^{n-2j}} 
\left(q_j^{(\mu,\beta_{n-2j},M_{n-2j})},\,q_k^{(\mu,\beta_{n-2j},M_{n-2j})}\right)_{\mu,\beta_{n-2j}}^{M_{n-2j}} \\
& \times \delta_{n-2j,m-2k} \, \delta_{\nu,\eta}.
\end{align*}
And the result follows from the orthogonality of the polynomial
$q_j^{(\mu,\beta_{n-2j},M_{n-2j})}$.
\end{proof}

\bigskip

\section{The Uvarov modification of Jacobi polynomials}
\setcounter{equation}{0}

In this section we will consider the study of several properties of the univariate orthogonal polynomials involved
in (\ref{eq:op1d}).

Let $(\cdot,\cdot)_{\alpha,\beta}^M$ be the inner product defined in (\ref{eq:op1d})
\begin{equation} \label{M-ip}
(f,g)_{\alpha,\beta}^M = \int_{-1}^1 f(t)\, g(t) (1-t)^\alpha (1+t)^\beta dt + M f(1) g(1)
\end{equation}
where $M$ is a positive real number.
Let $\{q_k^{(\alpha, \beta;M)}(t)\}_{k\geqslant 0}$ be the orthogonal polynomials with respect to \eqref{M-ip}  
having the same leading coefficient as the Jacobi polynomial
$P_k^{(\alpha, \beta)}$, and denote by  $\tilde{K}_{k}^{(\alpha, \beta, M)}(t,u)$
the corresponding kernels.

\bigskip

Following Uvarov (\cite{Uvarov1969}) these univariate orthogonal polynomials can be expressed in terms of the classical Jacobi polynomials as the first identity in the following lemma shows. Some of these 
properties are very well known (see \cite[p. 131]{Nevai1979}) but we include here a sketch of the
proof for the sake of completeness.

\begin{lemma}\label{Uvarov}
For $\alpha,\beta>-1$, it holds

\begin{align}
\label{M-polynomials}
&q_{k}^{(\alpha, \beta, M)}(t) = P_{k}^{(\alpha, \beta)}(t) - \dfrac{M P_{k}^{(\alpha, \beta)}(1)}{1 + M K_{k-1}^{(\alpha, \beta)}(1,1)} K_{k-1}^{(\alpha, \beta)}(1,t), \\
\label{M-norms}
&{\tilde h}_k^{(\alpha,\beta,M)} := (q_{k}^{(\alpha, \beta, M)},q_{k}^{(\alpha, \beta, M)})_{\alpha,\beta}^{M} =
h_k^{(\alpha,\beta)} \dfrac{1 + M K_{k}^{(\alpha, \beta)}(1,1)}{1 + M K_{k-1}^{(\alpha, \beta)}(1,1)},\\
\label{M-kernels} 
&\tilde{K}_{k}^{(\alpha, \beta, M)}(t,u) = K_{k}^{(\alpha, \beta)}(t,u) - \dfrac{M K_{k}^{(\alpha, \beta)}(1,t)K_{k}^{(\alpha, \beta)}(1,u)}{1 + M K_{k}^{(\alpha, \beta)}(1,1)}, 
\end{align}
In particular
\begin{align} \label{M-kernels-one}
&\tilde{K}_{k}^{(\alpha, \beta, M)}(1,1) = \dfrac{ K_{k}^{(\alpha, \beta)}(1,1)}{1 + M K_{k}^{(\alpha, \beta)}(1,1)}.
\end{align}

\end{lemma}

\bigskip

\begin{proof} Expand $q_k^{(\alpha,\beta,M)}$ in terms of Jacobi polynomials,
$$
q_k^{(\alpha,\beta,M)}(t)=\sum_{i=0}^k b_{k,i} P_{i}^{(\alpha, \beta)}(t),
$$
where $b_{k,k} = 1$. For $i=0, \ldots, k-1$ we have
$$ h_{i}^{(\alpha,\beta)} b_{j,i} =(q_k^{(\alpha,\beta,M)},
P_{i}^{(\alpha, \beta)})_{\alpha,\beta} =
- M q_k^{(\alpha,\beta,M)}(1) P_{i}^{(\alpha, \beta)}(1),
$$
thus
$$
q_{k}^{(\alpha, \beta, M)}(t) = P_{k}^{(\alpha, \beta)}(t) - M q_{k}^{(\alpha, \beta, M)}(1) K_{k-1}^{(\alpha, \beta)}(1,t),
$$
which gives
\begin{equation} \label{q(1)}
q_{k}^{(\alpha, \beta, M)}(1) = \frac{P_{k}^{(\alpha, \beta)}(1)}{1 + M K_{k-1}^{(\alpha, \beta)}(1,1)},
\end{equation}
and therefore \eqref{M-polynomials} holds.

Relation \eqref{M-norms} follows again from \eqref{q(1)} since
\begin{align*}
{\tilde h}_k^{(\alpha,\beta,M)} &= (q_{k}^{(\alpha, \beta, M)},q_{k}^{(\alpha, \beta, M)})_{\alpha,\beta}^{M} =
(q_{k}^{(\alpha, \beta, M)},P_{k}^{(\alpha, \beta)})_{\alpha,\beta}^{M} = \\
 &= (q_{k}^{(\alpha, \beta,M)},P_{k}^{(\alpha, \beta)})_{\alpha,\beta} + M  
q_{k}^{(\alpha, \beta, M)}(1)P_{k}^{(\alpha, \beta)}(1) =\\
&= h_k^{(\alpha,\beta)} + \frac{M P_{k}^{(\alpha, \beta)}(1)^2}{1 + M K_{k-1}^{(\alpha, \beta)}(1,1)}
= h_k^{(\alpha,\beta)}\dfrac{1 + M K_{k}^{(\alpha, \beta)}(1,1)}{1 + M K_{k-1}^{(\alpha, \beta)}(1,1)}.
\end{align*}

Now, from \eqref{M-polynomials}, \eqref{M-norms} and the identity 
$$K_{k}^{(\alpha, \beta)}(t,1) - K_{k-1}^{(\alpha, \beta)}(t,1) = 
\frac{P_{k}^{(\alpha, \beta)}(t)P_{k}^{(\alpha, \beta)}(1)}{h_k^{(\alpha,\beta)}},$$
we can easily prove
\begin{align*}
\frac{q_{k}^{(\alpha, \beta,M)}(t) q_{k}^{(\alpha, \beta,M)}(u)}{{\tilde h}_k^{(\alpha,\beta,M)}} &= 
\frac{P_{k}^{(\alpha, \beta)}(t)P_{k}^{(\alpha, \beta)}(u)}{h_k^{(\alpha,\beta)}} - \\
&- \dfrac{M K_{k}^{(\alpha, \beta)}(1,t)K_{k}^{(\alpha, \beta)}(1,u)}{1 + M K_{k}^{(\alpha, \beta)}(1,1)}
+ \dfrac{M K_{k-1}^{(\alpha, \beta)}(1,t)K_{k-1}^{(\alpha, \beta)}(1,u)}{1 + M K_{k-1}^{(\alpha, \beta)}(1,1)}
\end{align*}
and a telescopic sum gives \eqref{M-kernels}.
\end{proof}


\section{The kernels}
\setcounter{equation}{0}

The main purpose of this section is to study the reproducing kernels associated to the orthogonal polinomials $Q_{j,\nu}^{n}(x)$. In particular, we will establish relations with the classical kernels on the unit ball. 
The $n$-th classical \textbf{kernel} on the ball 
is usually defined as the polynomial
$$
\mathbb{K}_n(x,y) := \sum_{m=0}^n \sum_{j=0}^{\lfloor \frac{m}{2}\rfloor}
\sum_{\nu = 1}^{a_{m-2j}^d} \frac{P_{j,\nu}^m(x)P_{j,\nu}^m(y)}{H^m_{j,\nu}}.
$$
Our next result provides a representation of the $d$-variable kernels in terms of the univariate
Jacobi kernels.

\begin{theorem} \label{thm-5.1}
Let $d \geqslant 3$, for $x, y \in \mathbb{B}^d$ we have
\begin{align} \label{eqn-5.1}
\mathbb{K}_n(x,y) =& \frac{1}{c_{\mu}^d} \,\sum_{k=0}^{n} 
K^{(\mu,k+\delta)}_{\left\lfloor\frac{n-k}{2}\right\rfloor}(2r^2-1,2s^2-1)\,\times \, (2\,r\,s)^{k}\,
\frac{k+\delta}{\delta} \,C_{k}^{\delta}(\langle \xi, \varrho \rangle),
\end{align}
where $x=r \,\xi$, $y=s \,\varrho$, $r=\|x\|$, $s= \|y\|$, $\xi, \varrho \in \mathbb{S}^{d-1}$, $\delta = (d-2)/2$ and $C_{k}^{\delta}$ are the Gegenbauer polynomials (\cite[(4.7.1)
in p. 80]{Szego1975}). 

For $d= 2$, \eqref{eqn-5.1} reduces to
\begin{align} \label{eqn-5.2}
\mathbb{K}_n(x,y) =& \frac{1}{c_{\mu}^d} \,\sum_{k=0}^{n} 
K^{(\mu,k+\delta)}_{\left\lfloor\frac{n-k}{2}\right\rfloor}(2r^2-1,2s^2-1)\,\times \, 2^{k+1}\,(r\,s)^{k}\,
 \,T_{k}(\langle \xi, \varrho \rangle),
\end{align}
where $T_k$ are the first kind Chebyshev polynomials  (\cite[p. 38]{Szego1975}).
\end{theorem}

\begin{proof} 
For $n \in \mathbb{N}_0$ and $0 \le j \le n/2$, let $\{Y_\nu^{n-2j}: 1\le \nu\le a_{n-2j}^d\}$ denote
an orthonormal basis for $\mathcal{H}_{n-2j}^d$.
In spherical-polar coordinates, $x = r\xi$ and $y = s\varrho$, since $Y_\nu^{m-2j}$ is 
homogeneous of degree $m-2j$ we get
\begin{align*}
\mathbb{K}_n(x,y) =& \sum_{m=0}^n \sum_{j=0}^{\lfloor \frac{m}{2}\rfloor}
\sum_{\nu = 1}^{a_{m-2j}^d} \frac{P_{j,\nu}^m(x)P_{j,\nu}^m(y)}{H^m_{j,\nu}}\\
=& \sum_{m=0}^n \sum_{j=0}^{\lfloor \frac{m}{2}\rfloor}
\frac{1}{H^m_{j,\nu}} P_{j}^{(\mu, m-2j + \delta)}(2\,\|x\|^2 -1)\,P_{j}^{(\mu, m-2j + \delta)}(2\,\|y\|^2 -1)
\\ 
 &\times \sum_{\nu = 1}^{a_{m-2j}^d}  Y_\nu^{m-2j}(x) \, Y_\nu^{m-2j}(y) \\
=& \frac{1}{c_{\mu}^d} \sum_{m=0}^n \sum_{j=0}^{\lfloor \frac{m}{2}\rfloor}
\frac{2^{m-2j}}{ h_j^{(\mu,m-2j+\delta)}} P_{j}^{(\mu, m-2j + \delta)}(2r^2 -1)\,P_{j}^{(\mu, m-2j + \delta)}(2s^2 -1)\\ 
 &\times (rs)^{m-2j} \sum_{\nu = 1}^{a_{m-2j}^d}  Y_\nu^{m-2j}(\xi) \, Y_\nu^{m-2j}(\varrho).
\end{align*}
Making use of  the addition formula of spherical harmonics for $d\geqslant 3$ (see \cite[p. 9]{DaiXu2013})
$$
\sum_{\nu=1}^{a_k^d} 
Y_{\nu}^k(\xi) Y_{\nu}^k(\varrho) = 
\frac{k+\delta}{\delta} \,C_{k}^{\delta}(\langle \xi, \varrho \rangle)
$$ 
we have
\begin{align*}
\mathbb{K}_n(x,y) =& \frac{1}{c_{\mu}^d} \sum_{m=0}^n \sum_{j=0}^{\lfloor \frac{m}{2}\rfloor}
\frac{1}{ h_j^{(\mu,m-2j+\delta)}} P_{j}^{(\mu, m-2j + \delta)}(2r^2 -1)\,P_{j}^{(\mu, m-2j + \delta)}(2s^2 -1)\\ 
 &\times (2rs)^{m-2j} \frac{m-2j+\delta}{\delta} \,C_{m-2j}^{\delta}(\langle \xi, \varrho \rangle).
\end{align*}
Now, make $k = m-2j$ to change the order in the double sum
\begin{align*}
\mathbb{K}_n(x,y) =& \frac{1}{c_{\mu}^d} \sum_{k=0}^n \sum_{j=0}^{\left\lfloor\frac{n-k}{2}\right\rfloor}
\frac{1}{ h_j^{(\mu,k+\delta)}} P_{j}^{(\mu, k + \delta)}(2r^2 -1)\,P_{j}^{(\mu, k+ \delta)}(2s^2 -1)\\ 
 &\times (2rs)^{k} \frac{k+\delta}{\delta} \,C_{k}^{\delta}(\langle \xi, \varrho \rangle),
\end{align*}
and therefore
\begin{align*}
\mathbb{K}_n(x,y) =& \frac{1}{c_{\mu}^d} \,\sum_{k=0}^{n} 
K^{(\mu,k+\delta)}_{\left\lfloor\frac{n-k}{2}\right\rfloor}(2r^2-1,2s^2-1)\,\times \, (2\,r\,s)^{k}\,
\frac{k+\delta}{\delta} \,C_{k}^{\delta}(\langle \xi, \varrho \rangle).
\end{align*}

The case $d=2$ follows from the limit relation
$$
\lim_{\delta \to 0} \frac{k+\delta}{\delta} \,C_{k}^{\delta}(t)
= 2 \, T_k(t),
$$
(see \cite[(4.7.8) in p. 80]{Szego1975}).
\end{proof}

\bigskip

In a similar way, we define the $n$-th kernel associated to the polynomials $Q_{j,\nu}^m(x)$ as 
$$
\tilde{\mathbb{K}}_n(x,y) := \sum_{m=0}^n \sum_{j=0}^{\lfloor \frac{m}{2}\rfloor}
\sum_{\nu = 1}^{a_{m-2j}^d} \frac{Q_{j,\nu}^m(x)Q_{j,\nu}^m(y)}{\tilde{H}^m_{j,\nu}}.
$$
Proceeding as in Theorem~\ref{thm-5.1} we can obtain a representation of this kernels 
in terms of the univariate kernels associated to the Uvarov modifications. Thus, for $x, y \in \mathbb{B}^d$, 
$d \geqslant 3$, we have
\begin{align} \label{Ktilde_1}
\tilde{\mathbb{K}}_n(x,y) =& \frac{1}{c_{\mu}^d} \,\sum_{k=0}^{n} 
\tilde{K}^{(\mu,k+\delta,M_k)}_{\left\lfloor\frac{n-k}{2}\right\rfloor}(2r^2-1,2s^2-1)\,\times \, (2\,r\,s)^{k}\,
\frac{k+\delta}{\delta} \,C_{k}^{\delta}(\langle \xi, \varrho \rangle).
\end{align}
For $d= 2$
\begin{align} \label{Ktilde_1_d2}
\tilde{\mathbb{K}}_n(x,y) =& \frac{1}{c_{\mu}^d} \,\sum_{k=0}^{n} 
\tilde{K}^{(\mu,k+\delta,M_k)}_{\left\lfloor\frac{n-k}{2}\right\rfloor}(2r^2-1,2s^2-1)\,\times \, 
2^{k+1}\,(r\,s)^{k} \,T_{k}(\langle \xi, \varrho \rangle).
\end{align}

\bigskip

Finally, from \eqref{M-kernels} we derive a formula connecting both kernels in terms of the
classical Jacobi kernels.

\bigskip

\begin{prop}\label{prop 5.2}
Let $x=r \,\xi$, $y=s \,\varrho$, $r=\|x\|$, $s= \|y\|$, $\xi, \varrho \in \mathbb{S}^{d-1}$.
For $n\geqslant0$ and $d\geqslant 3$, we get
\begin{align*} 
& \mathbb{K}_n(x,y) - \tilde{\mathbb{K}}_n(x,y) = \\
& \frac{1}{c_{\mu}^d} \,\sum_{k=0}^{n}
\frac{
M_k K^{(\mu,k+\delta)}_{\left\lfloor\frac{n-k}{2}\right\rfloor}(2r^2-1,1) \, K^{(\mu,k+\delta)}_{\left\lfloor\frac{n-k}{2}\right\rfloor}
(2s^2-1,1)}{1 + M_k K^{(\mu,k+\delta)}_{\left\lfloor\frac{n-k}{2}\right\rfloor}(1,1)} 
\times \,2^{k}\, (r\,s)^{k}\,
\frac{k+\delta}{\delta} \,C_{k}^{\delta}(\langle \xi, \varrho \rangle). \nonumber
\end{align*}
For $d=2$
\begin{align*} 
& \mathbb{K}_n(x,y) - \tilde{\mathbb{K}}_n(x,y) = \\
& \frac{1}{c_{\mu}^d} \,\sum_{k=0}^{n}
\frac{
M_k K^{(\mu,k+\delta)}_{\left\lfloor\frac{n-k}{2}\right\rfloor}(2r^2-1,1) \, K^{(\mu,k+\delta)}_{\left\lfloor\frac{n-k}{2}\right\rfloor}
(2s^2-1,1)}{1 + M_k K^{(\mu,k+\delta)}_{\left\lfloor\frac{n-k}{2}\right\rfloor}(1,1)} 
\times \,2^{k+1}\,(r\,s)^{k}\,
 \,T_{k}(\langle \xi, \varrho \rangle). \nonumber
\end{align*}
\end{prop}

\bigskip

\section{Asymptotics for Christoffel functions}
\setcounter{equation}{0}

In this section we shall show some asymptotic results for the Christoffel functions.
We must restrict ourselves to the case $\mu \geqslant-\frac{1}{2}$ because of existing
asymptotics for Christoffel functions on the ball have only been
established for this range of values. Our results include asymptotics for the interior 
of the ball as well as for its boundary.

On the boundary of the ball, we recover the value of the mass from the asymptotic of the
Christoffel functions.

\begin{theorem}\label{th1}
Assume that $\mu \geqslant-\frac{1}{2}$. For $\left\Vert x\right\Vert =1$, we get
\begin{equation}\label{ass1}
\lim_{n\rightarrow \infty }\frac{\tilde{\mathbb{K}}_n(x,x)}{\binom{n + d -1}{n}}=\frac{2}{\lambda}.
\end{equation}
\end{theorem}

\begin{proof}
From \eqref{Ktilde_1} and \eqref{M-kernels-one}, for $\left\Vert x\right\Vert =1$ we deduce
\begin{align*}
\tilde{\mathbb{K}}_n(x,x) =& \frac{1}{c_{\mu}^d} \,\sum_{k=0}^{n} 
\frac{K^{(\mu,k+\delta)}_{\left\lfloor\frac{n-k}{2}\right\rfloor}(1,1)}
{1 + M_k K^{(\mu,k+\delta)}_{\left\lfloor\frac{n-k}{2}\right\rfloor}(1,1)}\,
\times \, 2^{k}\,
\frac{k+\delta}{\delta} \,C_{k}^{\delta}(1),
\end{align*}
then, writing 
$$
\delta = \frac{d-2}{2}, \quad M_k =  \frac{\lambda}{c_{\mu}^d} 2^k \textrm{ and }
C_{k}^{\delta}(1) = \binom{k+2\delta-1}{k}
$$
we get
\begin{align*}
\tilde{\mathbb{K}}_n(x,x) =& \frac{1}{\delta} \,\sum_{k=0}^{n} 
\frac{2^k K^{(\mu,k+\delta)}_m(1,1)}
{c_{\mu}^d  + \lambda 2^k K^{(\mu,k+\delta)}_m(1,1)}\,
(k+\delta) \,\binom{k+d-3}{k},
\end{align*}
where we denote $m = {\left\lfloor\frac{n-k}{2}\right\rfloor}$.

Let us split the above sum in two parts
$$
\tilde{\mathbb{K}}_n(x,x) = \frac{1}{\delta} \left( S_{1,n} + S_{2,n} \right), 
$$
with 
\begin{align*}
S_{1,n} &= \sum_{k=0}^{n-\left\lfloor \log{n} \right\rfloor} 
\frac{2^k K^{(\mu,k+\delta)}_m(1,1)}
{c_{\mu}^d  + \lambda 2^k K^{(\mu,k+\delta)}_m(1,1)}\,
(k+\delta) \,\binom{k+d-3}{k},\\
S_{2,n} &= \sum_{k=n-\left\lfloor \log{n} \right\rfloor + 1}^{n} 
\frac{2^k K^{(\mu,k+\delta)}_m(1,1)}
{c_{\mu}^d  + \lambda 2^k K^{(\mu,k+\delta)}_m(1,1)}\,
(k+\delta) \,\binom{k+d-3}{k},
\end{align*}
In order to estimate $S_{1,n}$ we use relation \eqref{K(1,1)} to obtain
$$
2^k K^{(\mu,k+\delta)}_m(1,1) = \frac{2^{-\mu-\delta-1}}{\Gamma(\mu+1)\Gamma(\mu+2)} 
\frac{\Gamma(m + k + \mu + \delta + 2)}{\Gamma(m + k + \delta + 1)}. 
\frac{\Gamma(m + \mu+2)}{\Gamma(m+1)}
$$
Now, we can use the following consequence of Stirling's formula: for fixed $a,b$, as $x\rightarrow \infty$,
$$
\frac{\Gamma (x+b)}{\Gamma ( x+a) }=x^{b-a}( 1+o(1) ),
$$
in this way
\begin{align*} 
2^k K^{(\mu,k+\delta)}_m(1,1) = \frac{2^{-\mu-\delta-1}}{\Gamma(\mu+1)\Gamma(\mu+2)}
(m+k)^ {\mu+1} m^{\mu+1} (1 + o(1)).
\end{align*}
Hence it follows
\begin{align} \label{K_m_11}
\frac{2^k K^{(\mu,k+\delta)}_m(1,1)}{c_{\mu}^d  + \lambda 2^k K^{(\mu,k+\delta)}_m(1,1)} = 
\frac{1}{\lambda}  (1 + o(1))
\end{align}
as $m \rightarrow \infty$.

Next, we estimate $(k+\delta)\binom{k+d-3}{k}$. For $0 \leqslant k \leqslant n$ we have
\begin{align}
\frac{1}{n^{d-2}} (k+\delta)\binom{k+d-3}{k} &= \frac{1}{(d-3)!}\left(\frac{k}{n}+\frac{d-2}{2n}\right)
\left(\frac{k}{n}+\frac{d-3}{n}\right)\ldots \left(\frac{k}{n}+\frac{1}{n}\right) \nonumber\\ 
&=\frac{1}{(d-3)!}\left(\left(\frac{k}{n}\right)^{d-2} + \frac{a_0}{n} \left(\frac{k}{n}\right)^{d-3}
+ \ldots + \frac{a_{d-3}}{n^{d-2}}\right) \nonumber\\
&\leqslant \frac{1}{(d-3)!}\left(\left(\frac{k}{n}\right)^{d-2} + \frac{C}{n} \right) \label{c-ineq}
\end{align}
where $C = a_0 + a_1 + \ldots + a_{d-3} > 0$ is independent of $k$ and $n$. Therefore
\begin{align*}
\frac{S_{1,n}}{n^{d-1}} &= 	\frac{1}{\lambda \delta n} 
\sum_{k=0}^{n-\left\lfloor \log{n} \right\rfloor} 
\frac{1}{(d-3)!} \left[\left(\frac{k}{n}\right)^{d-2} + O\left(\frac{1}{n}\right) \right]
\, \left(1 + o(1) \right)\\
&= 	\frac{2}{\lambda (d-2)! n} 
\sum_{k=0}^{n-\left\lfloor \log{n} \right\rfloor} 
 \left[\left(\frac{k}{n}\right)^{d-2} + O\left(\frac{1}{n}\right) \right]
\, \left(1 + o(1)\right).
\end{align*}
Next, since $0\leqslant k\leqslant n$, we have
$$
0 \leqslant \frac{1}{n}\sum_{n-\left\lfloor \log{n} \right\rfloor+1}^n   
 \left(\frac{k}{n}\right)^{d-2} \leqslant
\frac{1}{n}\sum_{n-\left\lfloor \log{n} \right\rfloor+1}^n   
 1 \leqslant  \frac{\log{n}}{n}
$$
which implies
$$
\lim_{n\to\infty} \frac{1}{n}\sum_{k=0}^{n-\left\lfloor \log{n} \right\rfloor}   
 \left(\frac{k}{n}\right)^{d-2} =
\lim_{n\to\infty} \frac{1}{n}\sum_{k=0}^{n}   
 \left(\frac{k}{n}\right)^{d-2} = \frac{1}{d-1},
$$
where the last equality follows from Silverman-Toeplitz theorem (see \cite[p. 25]{Wimp1981}).
In this way, we conclude
$$
\lim_{n\to\infty} \frac{S_{1,n}}{\binom{n+d-1}{n}} = \frac{2}{\lambda}.
$$

On the other hand
\begin{align*}
S_{2,n} &= \sum_{k=n-\left\lfloor \log{n} \right\rfloor + 1}^{n} 
\frac{2^k K^{(\mu,k+\delta)}_m(1,1)}
{c_{\mu}^d  + \lambda 2^k K^{(\mu,k+\delta)}_m(1,1)}\,
(k+\delta) \,\binom{k+d-3}{k}\\
&\leqslant \frac{1}{\lambda}\sum_{k=n-\left\lfloor \log{n} \right\rfloor + 1}^{n} 
(k+\delta) \,\binom{k+d-3}{k}
\end{align*}
Using again \eqref{c-ineq}, for $0 \leqslant k \leqslant n$ we have
\begin{align} \label{k+delta}
\frac{1}{n^{d-2}} (k+\delta)\binom{k+d-3}{k} &\leqslant \frac{1}{(d-3)!}\left(\left(\frac{k}{n}\right)^{d-2} + \frac{C}{n} \right) \leqslant C' 
\end{align}
where $C' > 0$ is a constant independent of $k$ and $n$. Hence, we deduce
$$
\frac{S_{2,n}}{n^{d-1}} \leqslant \frac{C'}{\lambda} \,\frac{\log{n}}{n},
$$
this implies
$$
\lim_{n\to\infty} \frac{S_{2,n}}{\binom{n+d-1}{n}} = 0,
$$
and \eqref{ass1} follows for $d\geqslant 3$.

In the case $d=2$ and $\|x\| = 1$, from \eqref{Ktilde_1_d2} we have
\begin{align*}
\tilde{\mathbb{K}}_n(x,x) =& 2 \, \sum_{k=0}^{n} 
\frac{2^k K^{(\mu,k)}_m(1,1)}
{c_{\mu}^d  + \lambda 2^k K^{(\mu,k)}_m(1,1)}\,
\end{align*}
then, proceeding in the same way, from \eqref{K_m_11} we conclude
\begin{align*}
\lim_{n \to \infty} \frac{1}{n+1} \tilde{\mathbb{K}}_n(x,x) = \frac{2}{\lambda}.
\end{align*}
\end{proof}

In the following proposition an estimate on the reproducing kernels that is uniform in $k$ is obtained,

\begin{lemma}\label{le6}

Fix $\mu >-1,\delta \geqslant0$. For $m=\left\lfloor \frac{n-k}{2}\right\rfloor \geqslant1,\,k\geqslant
0$, and $t\in \left[ -1,1\right],$
\begin{eqnarray*}
K_{m}^{(\mu,k+\delta)}(t,t)   
&\leqslant  &C( 1+t) ^{-k}\left( \left\lfloor\frac{n}{2}\right\rfloor +1\right)
\\
&~& \times \left( 1-t+\frac{1}{( \left\lfloor \frac{n}{2}\right\rfloor +1) ^{2}}\right)
^{-\mu -\frac{1}{2}}\left( 1+t+\frac{1}{(\left\lfloor \frac{n}{2}\right\rfloor+1
) ^{2}}\right) ^{-\delta -\frac{1}{2}}.
\end{eqnarray*}
Here $C$ depends on $\mu $ and $\delta $ but not on $k,n,t$.
\end{lemma}

\begin{proof} 
From the extremal properties of Christoffel functions,
for $k$ even, say $k=2\ell $, we have 
\begin{align*}
K_{m}^{(\mu,k+\delta)}(t,t)     &= \sup_{\deg (
P) \leqslant  m}\frac{P^{2}( t) }{\int_{-1}^{1}P^{2}(
s) \left( 1-s\right) ^{\mu }\left( 1+s\right) ^{k+\delta }ds}
\\
& =( 1+t) ^{-k}\sup_{\deg ( P) \leqslant  m}\frac{\left(
P( t) \left( 1+t\right) ^{\ell }\right)^{2}}{\int_{-1}^{1}\left(
P( s) \left( 1+s\right) ^{\ell }\right)^{2}\left( 1-s\right)
^{\mu }\left( 1+s\right) ^{\delta }ds}
\\
&\leqslant  ( 1+t) ^{-k}\sup_{\deg ( R) \leqslant  m+\ell}
\frac{R( t) ^{2}}{\int_{-1}^{1}R( s) ^{2}\left(
1-s\right) ^{\mu }\left( 1+s\right) ^{\delta }ds} \\
& = ( 1+t) ^{-k} K_{m + \left\lfloor \frac{k}{2}\right\rfloor }^{(\mu,\delta)}(t,t)  
\leqslant ( 1+t) ^{-k} K_{\left\lfloor \frac{n}{2}\right\rfloor }^{(\mu,\delta)}(t,t).
\end{align*}
We now use a result from Nevai's 1979 Memoir \cite[p. 108, Lemma 5]
{Nevai1979}, 
\begin{align*}
K_{\left\lfloor \frac{n}{2}\right\rfloor }^{(\mu,\delta)}(t,t) \leqslant &  C\left( \left\lfloor \frac{n}{2}\right\rfloor +1\right)
\\
&\times  \left( 1-t+\frac{1}{( \left\lfloor \frac{n}{2}\right\rfloor +1) ^{2}}\right)^{-\mu -\frac{1}{2
}}\left( 1+t+\frac{1}{( \left\lfloor \frac{n}{2}\right\rfloor +1)^{2}}\right)^{-\delta-\frac{1}{2}},
\end{align*}
and the result follows for $k= 2\ell$.

The case $k=2\ell +1$ can be deduced from the above reasoning as follows
\begin{align*}
K_{m}^{(\mu,2\ell+1+\delta)}(t,t) &\leqslant  
( 1+t)^{-2\ell} K_{\left\lfloor \frac{n}{2}\right\rfloor }^{(\mu,\delta+1)}(t,t) \\
& \leqslant C ( 1+t)^{-2\ell} \left( \left\lfloor \frac{n}{2}\right\rfloor +1\right)
\\
&\times  \left( 1-t+\frac{1}{( \left\lfloor \frac{n}{2}\right\rfloor +1) ^{2}}\right)^{-\mu -\frac{1}{2
}}\left( 1+t+\frac{1}{( \left\lfloor \frac{n}{2}\right\rfloor +1)^{2}}\right)^{-\delta-\frac{3}{2}}\\
& \leqslant C ( 1+t)^{-2\ell-1} \left( \left\lfloor \frac{n}{2}\right\rfloor +1\right)
\\
&\times  \left( 1-t+\frac{1}{( \left\lfloor \frac{n}{2}\right\rfloor +1) ^{2}}\right)^{-\mu -\frac{1}{2
}}\left( 1+t+\frac{1}{( \left\lfloor \frac{n}{2}\right\rfloor +1)^{2}}\right)^{-\delta-\frac{1}{2}},
\end{align*}
since for $t\in [-1,1]$ we get $0 \leqslant 1+t < 1 + t + \frac{1}{\left( \left\lfloor \frac{n}{2}\right\rfloor +1\right)^{2}}$.
\end{proof}

\bigskip

\begin{theorem}\label{th2}

For $r=\left\Vert x\right\Vert < 1$, we have
\begin{eqnarray*}
0 &<& \mathbb{K}_{n}( x,x) -\tilde{\mathbb{K}}_{n}( x,x)\\
&\leqslant  &  Cn^{d-1}
\log n \left( 2 (1-r^{2}) +\frac{4}{n^{2}}\right)^{-\mu -\frac{1}{2}}
\left( 2r^{2}+\frac{4}{n^{2}}\right) ^{-\delta -\frac{1}{2}}.
\end{eqnarray*}
Here $C$ is independent of $n$ and $x$. Consequently if $\mu \geqslant-\frac{1}{2}$, uniformly for $x$ in compact
subsets of $\left\{ x:0<\left\Vert x\right\Vert <1\right\} $,
\begin{equation}\label{ass2}
\lim_{n\rightarrow \infty }\tilde{\mathbb{K}}_{n}( x,x) / \binom{n+d}{d}
=\frac{1}{\sqrt{\pi }}
\frac{\Gamma ( \mu +1) \Gamma ( \frac{d+1}{2}) }{\Gamma ( \mu +\frac{d}{2}+1) }
\left(1-\left\Vert x\right\Vert ^{2}\right) ^{-\frac{1}{2}-\mu }.
\end{equation}
This last limit also holds for $x=0$.
\end{theorem}

\begin{proof}

Let us consider $x \in D$, with $D$ a compact subset of $\mathbb{R}^d$.
With $t=2r^{2}-1$ and $r=\|x\|$ we can assume that $t\leqslant  1-\eta$ for some $\eta >0$. 
Then, from Christoffel--Darboux formula and using the
convention $p_{j}=p_{j}^{( \mu ,k+\delta ) }$ for orthonormal Jacobi polynomials, we have
\begin{align}
\left\vert K_{m}^{( \mu ,k+\delta )}( t,1) \right\vert  &=\frac{\gamma _{m}}{
\gamma _{m+1}}\left\vert \frac{p_{m+1}( t) p_{m}( 1)
-p_{m}( t) p_{m+1}( 1) }{t-1}\right\vert  
\nonumber \\
&\leqslant  \frac{C}{2}\frac{\sqrt{p_{m}^{2}( t) +p_{m+1}^{2}(
t) }\sqrt{p_{m}^{2}( 1) +p_{m+1}^{2}( 1) }}{\eta} \nonumber
\\ \label{0cotakm}
&\leqslant  \frac{C}{2\eta }K_{m+1}^{( \mu ,k+\delta )}( t,t) ^{1/2}\sqrt{%
p_{m}^{2}( 1) +p_{m+1}^{2}( 1) },
\end{align}%
where $\gamma_m$ is the leading coefficient of $p_{m}$.

Next, we note that given any real number $a$, there exists $C_{a}>1$ such
that for all $x$ with $\min ( x,x+a) \geqslant1,$
\[
C_{a}^{-1}x^{a}\leqslant  \frac{\Gamma ( x+a) }{\Gamma ( x) }\leqslant  C_{a}x^{a}.
\]
This follows from Stirling's formula and the positivity and continuity of
$\frac{\Gamma ( x+a) }{\Gamma ( x) }$ for this range of $x$. Then
from (\ref{jac-norm}) and \eqref{normJ} we get
$$
\left( p_{m}^{( \mu ,k+\delta ) }( 1) \right)^2
\leqslant   C\frac{( m+k) ^{\mu + 1}\,m^{\mu}}{2^{k}}.
$$
Substituting these bounds into (\ref{0cotakm}) for $m\geqslant 1$ we  have
\begin{align*}
\left( K_{m}^{( \mu ,k+\delta ) }( t,1) \right)^2 &\leqslant  C
K_{m+1}( t,t) \, \frac{(m+k)^{\mu + 1}\,m^{\mu}}{2^{k}}.
\end{align*}
In the same way, using \eqref{K(1,1)}, we deduce
\begin{align*}
2^k \, K_{m}^{( \mu ,k+\delta ) }( 1,1) &\geqslant  C' \, (m+k)^{\mu + 1}\,m^{\mu+1}.
\end{align*}
And therefore we conclude
\begin{align}
\frac{2^{k} \left( K_{m}^{( \mu ,k+\delta ) }(t,1)\right)^2}{c_{\mu}^d + \lambda 2^{k}K_{m}^{( \mu ,k+\delta ) }(1,1)}  &\leqslant  
\frac{C K_{m+1}^{( \mu ,k+\delta ) }(t,t) (m+k)^{\mu+1} m^{\mu}}{c_{\mu}^d + \lambda C' (m+k)^{\mu+1} m^{\mu+1} }
\nonumber \\
&\leqslant C \frac{K_{m+1}^{( \mu ,k+\delta ) }(t,t)}{m+1}.\label{cotafkm}
\end{align}
This bound holds also for $m=0$, thought it is obtained in a simpler way since $K_{0}$ is a constant.

Then, for $ d \geqslant 3$, we have
$$
0 \leqslant \mathbb{K}_{n}( x,x) -\tilde{\mathbb{K}}_{n}( x,x) \leqslant  C\sum_{k=0}^{n}2^{k}( k+\delta )
\binom{k+d-3}{k} r^{2k} \frac{K_{m+1}^{( \mu ,k+\delta ) }( t,t) }{m+1}.
$$
Using the bound \eqref{k+delta} and denoting $t=2r^{2}-1$ and $m=\left\lfloor\frac{n-k}{2}\right\rfloor$, 
Lemma~\ref{le6} gives
\begin{align*}
0 \leqslant \mathbb{K}_{n}( x,x) &-\tilde{\mathbb{K}}_{n}( x,x) \leqslant 
\\
\leqslant &  C \, n^{d-2} \sum_{k=0}^{n}
 \frac{\left\lfloor \frac{n}{2}\right\rfloor+2 }{\left\lfloor \frac{n-k}{2}\right\rfloor+1}
\\
&\times \left( 1-t+\frac{1}{(\left\lfloor \frac{n}{2}\right\rfloor +2)^{2}}\right)^{-\mu -\frac{1}{2}}
\left( 1+t+\frac{1}{(\left\lfloor \frac{n}{2}\right\rfloor +2) ^{2}}\right)^{-\delta -\frac{1}{2}}
\\
\leqslant & C n^{d-1}\left( 1-t+\frac{4}{n^{2}}\right) ^{-\mu -\frac{1}{2}}
\left( 1+t+\frac{4}{n^{2}}\right)^{-\delta -\frac{1}{2}}\sum_{k=0}^{n}
\frac{1}{\left\lfloor \frac{n-k}{2}\right\rfloor +1}
\\
\leqslant  & C n^{d-1}\log n \left( 2(1-r^{2}) +\frac{4}{n^{2}}\right)^{-\mu -\frac{1}{2}}
\left( 2r^{2}+\frac{4}{n^{2}}\right) ^{-\delta -\frac{1}{2}},
\end{align*}
where the last inequality follows from
\begin{align*}
\lim_{n\to\infty} \frac{1}{\log n} \sum_{k=0}^{n}\frac{1}{\left\lfloor \frac{n-k}{2}\right\rfloor +1} = 2.
\end{align*}
Obviously the above inequality holds also in the case $d=2$.

Finally,  \cite[Theorem 1.3]{KrooLubinsky2013A} { shows}
\begin{align*}
\lim_{n\rightarrow \infty }\mathbb{K}_{n}\left( x,x\right) /\binom{n+d}{d}
=& \frac{\omega _{\mu }W_{0}( x) }{\left( 1-\left\Vert
x\right\Vert ^{2}\right) ^{\mu }}\\
=&\frac{1}{\sqrt{\pi }}
\frac{\Gamma (\mu +1)
\Gamma ( \frac{d+1}{2}) }{\Gamma ( \mu +\frac{d}{2}+1) }
\left( 1-\left\Vert x\right\Vert ^{2}\right) ^{-\frac{1}{2}-\mu },
\end{align*}
uniformly for $x$ in compact subsets of the unit ball.  {Consequently}  $\mathbb{K}_{n}(x,x) $
grows like $n^{d}>>n^{d-1}\log n$,   { and clearly} (\ref{ass2}) follows.

Finally, in the case $x=0$, that is $r=0$, all the
terms in $\mathbb{K}_{n}(0,0) -\tilde{\mathbb{K}}_{n}(0,0)$ 
vanish except for $k=0$. If we write $m = \left\lfloor \frac{n}{2}\right\rfloor$
we get
\begin{align*}
\mathbb{K}_{n}(0,0) -\tilde{\mathbb{K}}_{n}(0,0) &= \frac{\lambda \left(K_m^{( \mu ,\delta ) }(-1,1)\right)^2}{c_{\mu}^d + \lambda K_m^{( \mu ,\delta ) }(1,1)} \\
& \leqslant \frac{C K_m^{( \mu ,\delta ) }(-1,-1) m^{2\mu+1}}{c_{\mu}^d + C' m^{2\mu+2}} 
\leqslant \frac{K_{m+1}^{( \mu ,\delta ) }(-1,-1)}{m+1}.
\end{align*}
Next,  Lemma~\ref{le6} {implies}
\begin{align*}
\mathbb{K}_{n}(0,0) -\tilde{\mathbb{K}}_{n}(0,0) \leqslant &  C \,  
\left( 2 +\frac{1}{(\left\lfloor \frac{n}{2}\right\rfloor +2)^{2}}\right)^{-\mu -\frac{1}{2}}
\left( \frac{1}{(\left\lfloor \frac{n}{2}\right\rfloor +2) ^{2}}\right)^{-\delta -\frac{1}{2}}
\\
\leqslant & \, C \, \left(\left\lfloor \frac{n}{2}\right\rfloor +2\right)^{2 \delta + 1} 
\leqslant \, C \, n^{d-1}
\end{align*}
and, therefore, \eqref{ass2} follows also in this case.
\end{proof}

\bigskip

\section*{Acknowledgements}

The authors thank MINECO of Spain and the European Regional Development Fund (ERDF) through grant
MTM2014--53171--P, and Jun\-ta de Andaluc\'{\i}a grant P11--FQM--7276 and research group FQM--384.

\end{document}